\documentclass[a4paper,12pt,oneside,reqno]{amsart}
\usepackage{amssymb,amsfonts,amsmath,amsthm,amscd}
\usepackage{graphicx, color}
\usepackage[normalem]{ulem}
\numberwithin{equation}{section}  

\newcommand{\beq}{\begin{equation}} 
\newcommand{\eeq}{\end{equation}} 
\newcommand{\bea}{\begin{aligned}}
\newcommand{\eea}{\end{aligned}}
\newcommand{\bdm}{\begin{displaymath}}
\newcommand{\edm}{\end{displaymath}}
\newcommand{\barr}{\begin{array}}
\newcommand{\earr}{\end{array}}
\newcommand{\ben}{\begin{enumerate}}
\newcommand{\een}{\end{enumerate}}
\newcommand{\bde}{\begin{description}}
\newcommand{\ede}{\end{description}}

\newtheorem{teor}{Theorem}
\newtheorem{prop}[teor]{Proposition}  
\newtheorem{lem}[teor]{Lemma}

\newcommand{\R}{\mathbb{R}}

\newcommand{\N}{\mathbb{N}}

\newcommand{\PP}{\mathbb{P}}
\newcommand{\E}{{\mathbb{E}}}

\newcommand{\defi}{\equiv} 

\newcommand{\be}{\beta}

\newcommand{\de}{\delta}

\newcommand{\s}{\sigma}

\begin{document}
\title[From REM to BRW]{From Derrida's random energy model to\\ 
branching random walks: from 1 to 3}
\author[Nicola Kistler]{Nicola Kistler}            
 \address{Nicola Kistler,
J.W. Goethe-Universit\"at
Frankfurt, Robert-Mayer-Str. 10, DE - 60325 Frankfurt.}
\email{kistler@math.uni-frankfurt.de}
\author[Marius Schmidt]{Marius A. Schmidt}            
 \address{Marius A. Schmidt, 
J.W. Goethe-Universit\"at Frankfurt, Robert-Mayer-Str. 10, DE - 60325 Frankfurt.}
\email{mschmidt@math.uni-frankfurt.de}

\begin{abstract}
We study the extremes of a class of Gaussian fields with in-built hierarchical structure. The number of scales in the underlying trees depends on a parameter  $\alpha \in \left[0,1\right]$: choosing $\alpha=0$ yields 
the random energy model by Derrida (REM), whereas $\alpha=1$ corresponds to the branching random walk (BRW). When the parameter $\alpha$ increases, the level of the maximum of the field decreases smoothly from the REM- to  the BRW-value. However, as long as $\alpha<1$ strictly, the limiting extremal process is always Poissonian. 
\end{abstract}


\date{\today}

\maketitle

\section{Introduction and main result} 
The Gaussian fields we consider are constructed as follows. Let $\alpha \in \left[0,1\right]$ and $N\in \N$. We refer to the parameter $N$ as the {\it size of the system}. For $j=1 \dots N^\alpha$ and $\s_j = 1 \dots \exp\left(N^{1-\alpha} \log 2\right)$, consider the vectors $\s = \left(\s_1, \dots, \s_{N^\alpha}\right)$. (We assume, without loss of generality, that $N$ and $\alpha$ are such that $N^\alpha$ and $N^{1-\alpha}$ are both integers). We refer to the indices $j= 1 \dots N^\alpha$ as {\it scales}, and to the labels $\s$ as {\it configurations}. The space of configurations is denoted by $\Sigma_N^{\left(\alpha\right)}$. Remark that, by construction, $\sharp \Sigma_N^{\left(\alpha\right)} = 2^N$. For scales $j\leq N^\alpha$ and $\left(\s_1, \dots, \s_j\right)$, consider independent centered Gaussian random variables $X_{\s_1, \dots, \s_j}^{\left(\alpha, j\right)}$ with variance $N^{1-\alpha}$ defined on some common probability space $\left(\Omega, \mathcal F, \PP\right)$. To given configuration $\s \in \Sigma_N^{\left(\alpha\right)}$ we associate the {\it energies} 
\beq \label{rf}
X_\s^{\left(\alpha, N\right)} \defi \sum_{j=1}^{N^\alpha} X_{\s_1, \dots, \s_j}^{\left(\alpha, j\right)}
\eeq
The collection $X^{\left(\alpha, N\right)} \defi \left\{ X_\s^{\left(\alpha, N\right)},  \s \in \Sigma_N^{\left(\alpha\right)} \right\}$ defines a centered Gaussian field with
\[
\text{var}\left[ X_\s^{\left(\alpha, N\right)} \right] = N, \quad \text{and}\quad \text{cov}\left[X_\s^{\left(\alpha, N\right)}, X_\tau^{\left(\alpha, N\right)}\right] = \left(\s \wedge \tau\right) N^{1-\alpha},
\] 
where $\s\wedge \tau \defi \inf\left\{j \leq N^{\alpha}: \left(\s_1, \dots, \s_{j}\right) = \left(\tau_1, \dots, \tau_j\right) \; \text{and}\; \s_{j+1} \neq \tau_{j+1} \right\}$. In spin glass terminology, $\s\wedge \tau$ is the {\it overlap} of the configurations $\s$ and $\tau$. In other words, the Gaussian field $X^{\left(\alpha, N\right)}$ is {\it hierarchically} correlated.  The parameter $\alpha$ governs the number of scales in the underlying "trees".  The choice $\alpha=0$ yields the celebrated REM of Derrida \cite{Derrida_REM}; in this case the tree consists of a single scale (only for this boundary case is the field uncorrelated). The choice $\alpha=1$ yields the (classical) BRW, also known as the directed polymer on Cayley trees \cite{derrida_spohn}: in this model, the number of scales grows linearly with the size of the system. In this sense, the fields $X^{\left(\alpha, N\right)}$  interpolate between REM and BRW (remark that these boundary cases are, within our class, the least resp. the most correlated fields). 
See Figure \ref{figone} below for a graphical representation.
\begin{figure}[h]
    \centering
\includegraphics[scale=0.35]{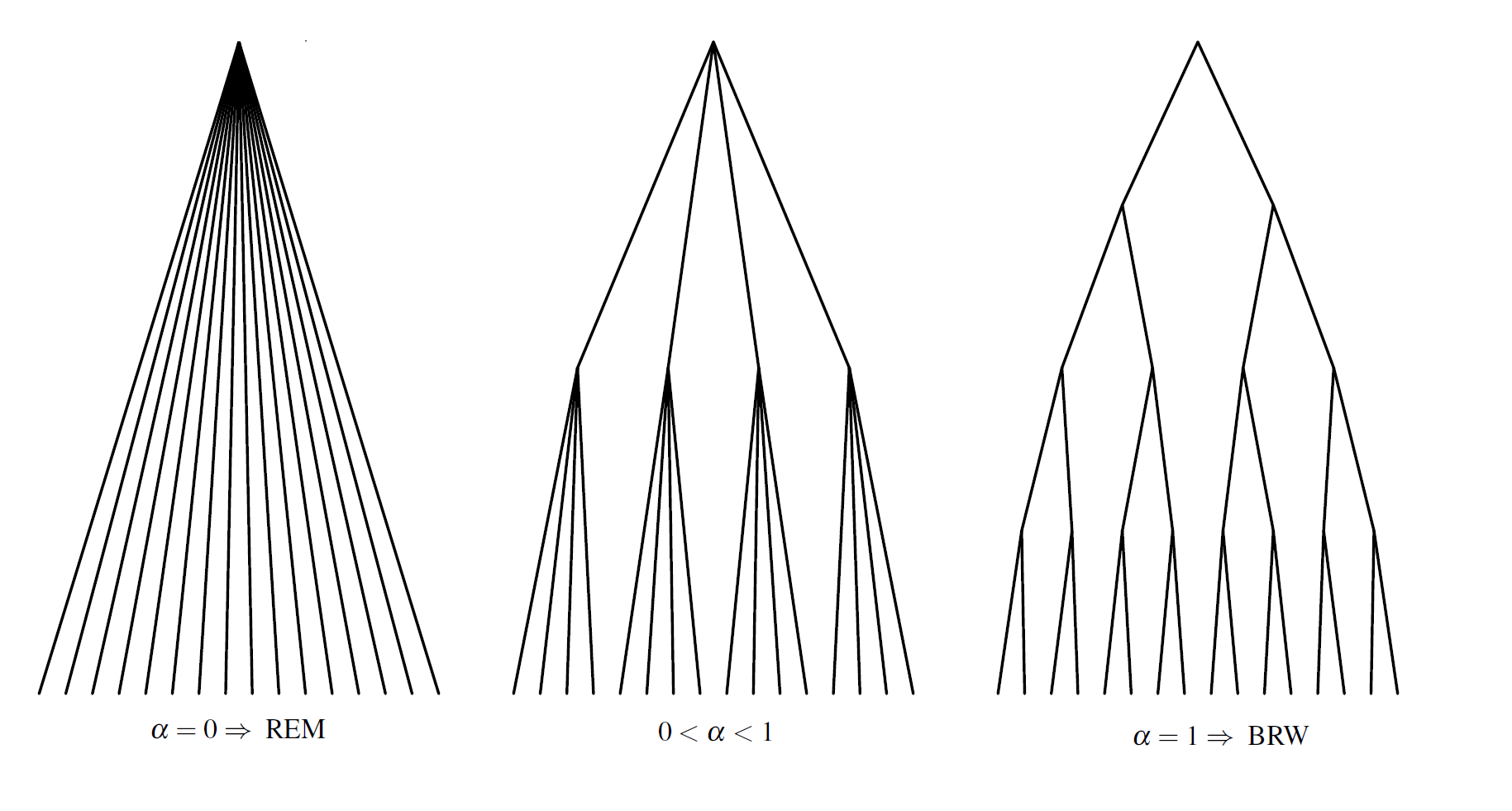}
    \caption{Trees interpolating between REM and BRW}
    \label{figone}
\end{figure}

\begin{center}

\end{center}

A fundamental question in the study of random fields concerns the behavior of the extreme values in the limit of large system-size. The case of independent random variables is simple, and completely understood, see e.g. the classic \cite{Leadbetter}. On the other hand, the study of the extremes of {\it correlated} random fields is a much harder question.  There is good reason to develop an extreme value theory for  Gaussian fields defined on trees: besides being typically amenable to a detailed analysis (see e.g. \cite{aidekon_et_al, abk_three, bovier_hartung, bovier_hartung_two, bovier_hartung_three, BovierKurkova_I, fang_zeitouni, madaule}),  Gaussian hierarchical fields should be some sort of  "universal attractors" in the limit of large system-size; this claim is a major pillar of the Parisi theory \cite{parisi} which has remained to these days rather elusive (see however \cite{kistler} and references therein for some recent advances). Our main result provides a characterization of the weak limit of the extremes of the hierarchical field \eqref{rf}.

\begin{teor} \label{main} Assume $\alpha \in \left[0, 1\right)$. Let 
\[ 
a_N^{\left(\alpha\right)} \defi \be_c N - \frac{1+2\alpha}{2 \be_c} \log N, \quad \text{where} \;\; \be_c \defi \sqrt{2\log 2},
\]
and consider the random Radon measure on the real line
\[
\Xi_N^{\left(\alpha\right)} \defi \sum_{\s \in \Sigma_N^{\left(\alpha\right)}} \de_{X_\s^{\left(\alpha, N\right)} - a_N^{\left(\alpha\right)}}\,.
\]
Then  $\;\Xi_N^{\left(\alpha\right)}$ converges weakly to a Poisson point process $\Xi$ of intensity $\mu(A) \defi \int_A  e^{-\be_c x} dx/\sqrt{2\pi}$. 
\end{teor}
Apart from the case $\alpha=0$, the picture depicted in Theorem \ref{main} seems to be new. There is good reason to leave out the case $\alpha=1$: to clarify this, and to shed further light on our main result, let us spend a few words. First, the theorem implies that $a_N^{\left(\alpha\right)}$ is the {\it level of the maximum} of the random field $X^{\left(\alpha, N\right)}$, and $\Xi_N^{\left(\alpha\right)}$ is then the {\it extremal process}. It steadily follows from the convergence of the extremal process that the maximum of the field, recentered by its level, weakly converges to a Gumbel distribution. As expected under the light of (say) Slepian's Lemma, the level of the maximum {\it decreases} when $\alpha$
(hence the amount of correlations) {\it increases}. However, this feature is only detectable at the level of the second order, logarithmic corrections; curiously, the pre-factor $1+2\alpha$ interpolates smoothly between the 
REM- and the BRW-values ("from 1 to 3"). Notwithstanding, as long as $\alpha< 1$ strictly, and in spite of what might look at first sight as severe correlations, all our models fall into the universality class of the REM, which is indeed characterized by convergence towards Poissonian extremal processes. In the boundary case of the BRW, the picture is only partially correct: the logarithmic correction is still given by $a_N^{\left(\alpha\right)}$ with $\alpha=1$, see \cite{addario, aidekon, bramson}, yet the weak limit of the maximum is no longer a Gumbel distribution \cite{lalley_sellke}, nor is the limiting extremal process a simple Poisson process \cite{aidekon_et_al, abk_three, brunet_derrida_two, madaule}. \\

We conclude this section with a sketch of the proof of our main result. A natural approach would be to choose $a_N^{(\alpha)}$ such that the expected number of extremal configurations in any given compact $A\subset \R$ is of order one in the large $N$-limit. However, with the level of the maximum as given by Theorem \ref{main},  classical Gaussian estimates steadily yield
\[
\E\left[ \Xi_N^{(\alpha)}(A) \right] = 2^N \int_A \exp\left[ - \left( x - a_N^{(\alpha)} \right)^2/ (2N) \right] \frac{dx}{\sqrt{2\pi N}} =  N^\alpha(1+o(1)) \qquad (N\to \infty), 
\]
which is exploding as soon as $\alpha > 0$ strictly. The reason for this is easily identified: by linearity of the expectation, we are completely omitting correlations, but these turn out to be strong enough to affect the level of the maximum. To overcome this problem, we rely on the multi-scale analysis which has emerged in the study of the extremes of branching Brownian motion (see e.g. \cite{kistler}). To formalize, we need some notation. First, for a given $\s \in \Sigma_N^{(\alpha)}$, we refer to the process 
\[
S^\s= (S^{\sigma}_k, k \leq N^\alpha), \quad S^{\sigma}_k \defi \sum\limits_{j\leq k} X_{\sigma_1,..\sigma_j}^{(\alpha, j)},
\] 
as the {\it path} of a configuration. (The process $S^\s$ is a random walk with Gaussian increments, i.e. a discrete Brownian motion). We refer to any function $F_N:\,  \{0\dots N^\alpha\} \to \R$,  $k \mapsto F_N(k),$ as {\it barrier}. Given a barrier $F_N$, we denote by 
$$ \Xi_{N, F_N}^{\left(\alpha\right)} \defi \sum\limits_{\sigma\in\Sigma_N}\delta_{X_\sigma^{(\alpha, N)}-a_N^{(\alpha)}}\mathbf{1}_{\left\{ S^\sigma_k \leq F_N(k) \; \text{for all}\;  k\in\left\{1,..,N^\alpha\right\} \right\}} $$
the {\it modified (extremal) process}. A key step in the proof is to identify a barrier $E_N$, see \eqref{barri} below for its explicit form, such that for any compact $A\subset \R$,
\beq \label{claim_equality}
\lim_{N\to \infty} \PP\left[ \Xi_{N}^{(\alpha)}(A) = \Xi_{N, E_N}^{(\alpha)}(A)  \right] = 1.
\eeq
This naturally entails that the weak limit of the extremal process and that of the modified process must coincide (provided one of the two exists). We will thus focus our attention on the modified process $\Xi_{N, E_N}^{(\alpha)}$, thereby proving that {\it mean of the process} as well as its {\it avoidance functions} converge to the Poissonian limit as given by Theorem \ref{main}, to wit:
\beq \label{claim_mean}
 \lim_{N\to \infty} \E\left[ \Xi_{N, E_N}^{(\alpha)}(A) \right] = \mu(A)  \qquad \qquad \text{(Convergence of mean)}
\eeq
and
\beq \label{claim_avoidance}
\lim_{N\to \infty} \PP\left( \Xi_{N, E_N}^{\left(\alpha\right)}\left(A\right) = 0 \right) = \PP\left( \Xi(A) = 0 \right) \qquad \text{(Avoidance functions)}
\eeq
By \eqref{claim_mean} and \eqref{claim_avoidance}, it follows  from Kallenberg's theorem on Poissonian convergence \cite{kallenberg}, that the modified process weakly converges to the Poisson point process $\Xi$; but by \eqref{claim_equality}, the same must be true for the extremal process, settling the proof of Theorem \ref{main}. \\ 

The rest of the paper is devoted to the proof of \eqref{claim_equality}, \eqref{claim_mean} and \eqref{claim_avoidance}. Since $\alpha \in [0,1)$ is fixed throughout, we lighten notations by dropping 
the $\alpha$-dependence whenever no confusion can possibly arise, writing e.g. $\Sigma_N$ for $\Sigma_N^{(\alpha)}$,  $X_\s$ for $X_\s^{(\alpha, N)}$, $a_N$ for $a_N^{(\alpha)}$, etc.  \\

\noindent{\bf Acknowledgements.} This paper owes much to conversations with Bernard Derrida, who raised in particular the question whether models with an increasing number of scales can provide (possibly quantitative) insights into the fractal structure of the extremal process of BRW/branching Brownian motion. Unfortunately, our main result shows that this is not the case, at least as long as $\alpha< 1$. It is tempting to believe that letting $\alpha$ depend on the size of the system (i.e. $\alpha(N) \to 1, \, \text{as}\, N\uparrow \infty$) gives rise to more interesting extremal processes.

\section{Barriers, and the modified processes} \label{two}
The goal of this section is to construct the barrier $E_N$ to which we alluded in the introduction, and to give a proof of \eqref{claim_equality} and \eqref{claim_mean}. In a first step, we construct a barrier which is not  
"optimal", but which provides important {\it a priori} information:

\begin{lem}\label{uenvelope} Consider the barrier 
\[ 
U_{N}(k) \defi \beta_c k N^{1-\alpha}+\ln\left(N\right), \; k =0,..,N^\alpha. 
\]
It then holds: 
$$
\lim_{N\to \infty} \PP \left(S^{\sigma}_k \leq U_{N}(k) \quad\forall k\in\left\{1,..,N^\alpha\right\}, \sigma\in\Sigma_N\right)= 1 \,.
$$
\end{lem}
\begin{proof}
By Markov inequality, and simple counting, it holds: 
\beq \bea \label{example}
& \PP\left(\exists \sigma\in \Sigma_N:\; \sum_{i\leq j} X_{\sigma_1,..,\sigma_i}^{\left(i\right)}> U_N(j), \; \text{for some}\; j\leq N^\alpha \right) \\
& \hspace{3.5cm} \leq \sum_{j\leq N^\alpha} \exp\left( j N^{1-\alpha} \ln 2\right) \PP\left( 
\sum_{i\leq j} X_{1,\dots,1}^{\left(j\right)}> \beta_c jN^{1-\alpha} +  \ln N
 \right)\,.
\eea \eeq
By classical Gaussian estimates, the probability on the r.h.s. above is at most 
\[
\frac{\sqrt{jN^{1-\alpha}}}{\sqrt{2\pi}\left(\beta_c jN^{1-\alpha} +  \ln N\right)}\exp\left[ -\frac{\left(\beta_c jN^{1-\alpha} +  \ln N\right)^2}{2jN^{1-\alpha}}\right].
\]
Using this, and straightforward estimates, we get 
\[
\eqref{example} \leq \exp\left[ \left( \frac{3\alpha - 1}{2 }- \be_c \right) \ln N\right],
\]
which is evidently vanishing in the large $N$-limit, since $\frac{3\alpha-1}{2} < \beta_c $. 
\end{proof}

The above Lemma immediately implies that the weak limit of the extremal process $\Xi_N$ and the weak limit of the modified process $\Xi_{N, U_N}$ must necessarily coincide (provided one of the two exists). We now identify conditions under which this remains true for barriers which lie even lower than $U_N$.

\begin{lem}\label{eenvelope}
Consider a barrier $F_N$  with the following properties: 
\begin{itemize} 
\item[i)] $F_N \leq U_N$, i.e. $F_{N}(k)\leq U_{N}(k)$ for all  $k$;
\item[ii)] for $A\subset \R$ compact, it holds:
$$\lim_{N\to \infty} \E\left[ \Xi_{N, F_N}\left(A\right) \right] = \lim_{N\to \infty} \E\left[ \Xi_{N, U_N}\left(A\right) \right]$$
\end{itemize}
Then the weak limits of $\Xi_{N, F_N}$ and $\Xi_{N, U_N}$ coincide (provided one of the two exists). 
\end{lem}
\begin{proof} The Lemma steadily follows from the claim
\beq \bea \label{difference_mean}
& \PP\left(\Xi_{N, U_N}\left(A\right) = \Xi_{N, F_N}\left(A\right)\right) \geq 1-\E\left[\Xi_{N, U_N}\left(A\right)-\Xi_{N, F_N}\left(A\right)\right]. 
\eea \eeq
The proof of \eqref{difference_mean} is straightforward. Simple rearrangements and subadditivity imply that the for probability of the complementary event, it holds: 
\[ \bea
& \PP\left(\Xi_{N, U_N}\left(A\right) \neq \Xi_{N, F_N}\left(A\right)\right) \\
& =\PP\left(\exists \sigma\in\Sigma_N: X_\sigma-a_N \in A,\; \forall_{j =1 \dots N^\alpha}:\; S^\sigma_j \leq U_{N}(j) \; \mbox{ but } \exists_{j=1\dots N^\alpha}\, : S^\sigma_j > F_{j,N} \right) \\
& \leq \sum\limits_{\sigma\in\Sigma_N}  \PP\left(\exists \sigma\in\Sigma_N: X_\sigma-a_N\in A,\forall_{j =1 \dots N^\alpha} :\; S^\sigma_j \leq U_{N}(j) \; \mbox{ but } \exists_{j=1\dots N^\alpha}\, : S^\sigma_j > F_{j,N} \right) \\
& = 2^N \PP\left(X_\sigma-a_N \in A, \forall_{j =1 \dots N^\alpha}: \;S^\sigma_j \leq U_{N}(j) \; \mbox{ but } \exists_{j=1\dots N^\alpha}\, : S^\sigma_j > F_{j,N} \right) \\
& = 2^N \PP\left(X_\sigma-a_N \in A,\forall_{j =1 \dots N^\alpha}: \; S^\sigma_j \leq U_{N}(j)  \right)-2^N \PP\left(X_\sigma-a_N \in A, \forall_{j =1 \dots N^\alpha}: \; S^\sigma_j \leq F_{N}(j) \right) \\
& = \E\left[\Xi_{N, U_N}(A)-\Xi_{N, F_N}(A)\right].
\eea \]
Building the complement, \eqref{difference_mean} immediately follows. 
\end{proof}
By the previous Lemma, and in view of a proof of the main theorem, it is crucial to identify conditions for which the mean(s) of the modified process(es) converge to a finite limit. This is done by 

\begin{prop}\label{fstmoment}
Consider a barrier of the form $F_N = U_N + f _N$, where $f_N$ is such that 
\begin{itemize}
\item[i)] $f_{N}(0)=f_{N}(N^\alpha)=0$ \vspace{0.2cm}
\item[ii)] $\sup\limits_{k\in\left\{1,..,N^\alpha\right\}}|f_{N}(k)|= o\left(N^{\frac{1-\alpha}{2}}\right)$ for $N\uparrow \infty$.
\end{itemize}
For $A\subset \R$ compact, and $\mu$ as in Theorem \ref{main}, it holds: 
\[
\lim_{N\rightarrow\infty} \E\left[\Xi_{N, F_N}\left(A\right)\right] = \mu(A)\,.
\]
\end{prop}

\begin{proof} By linearity of the expectation, and by conditioning on the "terminal event", 
\beq \bea \label{start}
& \E\left[\Xi_{N, F_N}\left(A\right)\right] = \\
& \qquad = 2^N \int\limits_A  \PP\left(\forall_{k\in\left\{1,..,N^\alpha\right\}}:\; S^\sigma_k \leq F_{N}(k)\;  \Big| X_\sigma-a_N =x\right)\PP\left(X_\sigma-a_N\in dx\right).
\eea \eeq
Let us  focus on the conditional probability: we first  write this as
\beq \bea \label{condprob}
& \PP\left(\forall_{k\in\left\{1,..,N^\alpha\right\}}:\; S^\sigma_k \leq F_{N}(k)\; \Big| X_\sigma-a_N=x\right)\\
&  \quad = \PP\left(\forall_{k\in\left\{1,..,N^\alpha\right\}}:\; S^\sigma_k - \frac{k}{N^\alpha}X_\sigma \leq F_{N}(k) -\frac{k}{N^\alpha}\left(a_N+x\right) \; \Big| X_\sigma-a_N =x\right)\,.
\eea \eeq
Inspection of the covariances shows that the Gaussian vector $\left(S_k^\s -\frac{k}{N^\alpha} X_\s, k=1\dots N^\alpha\right)$ is, in fact,  {\it independent} of $X_\s$. Using this, and rescaling by  $N^{-\frac{1-\alpha}{2}}$ yields
\beq 
\eqref{condprob} = \PP\left(\; \forall_{ k\in\left\{1,..,N^\alpha\right\}}:\; N^{-\frac{1-\alpha}{2}}\left[S^\sigma_k - \frac{k}{N^\alpha}S^\sigma_{N^\alpha}\right] \leq N^{-\frac{1-\alpha}{2}}\left[F_{N}(k) -\frac{k}{N^\alpha}\left(a_N+x\right)\right] \right).
\eeq
Again by inspection of the covariances, one immediately realizes that the law of the Gaussian vector 
$\left(N^{-\frac{1-\alpha}{2}}\left[S^\sigma_k - \frac{k}{N^\alpha}S^\sigma_{N^\alpha}\right], k = 0 \dots N^\alpha \right)$
is that of a (discrete) Brownian bridge of lifespan $N^\alpha$, starting and ending in $0$. To lighten notations, 
let $\left(B_{N^\alpha}(k), k \leq N^\alpha\right)$ be such a Brownian bridge, and shorten
\[
\widetilde F_{N}(k, x) \defi N^{-\frac{1-\alpha}{2}}\left[F_{N}(k) -\frac{k}{N^\alpha}\left(a_N+x\right)\right]\,.
\]
It thus holds: 
\[
\eqref{condprob} = \PP\left(\forall_{ k\in\left\{1,..,N^\alpha\right\}}:\; B_{N^\alpha}(k) \leq \widetilde F_{N}(k, x)\right)\,.
\] 
One immediately checks that within our choice of the barrier $F_N$, and since $\alpha < 1$ strictly, 
\[
\lim_{N\uparrow \infty} \sup_{k\leq N^\alpha, x \in A} \Big| \widetilde F_{N}(k, x) \Big| = 0\,.
\]
in which case it follows from the Lemmata in the Appendix that
\beq \bea \label{bbbridge}
\PP\left(\forall_{ k\in\left\{1,..,N^\alpha\right\}}:\; B_{N^\alpha}(k) \leq  \widetilde F_{N}(k, x) \right) & = \PP\left(
\; \forall_{ k\in\left\{1,..,N^\alpha\right\}}:\; B_{N^\alpha}(k) \leq 0 \right)(1+o(1)) \\
& = N^{-\alpha} \left(1+o(1)\right), 
\eea \eeq
uniformly for $x$ in compacts, and for $N\uparrow \infty$. Plugging this into \eqref{start}, we have 
\[
\E\left[\Xi_{N, F_N}\left(A\right)\right] = 2^N N^{-\alpha} \left(1+o(1)\right) \int\limits_A  \PP\left(X_\sigma-a_N\in dx\right)\,.
\]
The claim of the Proposition then immediately follows by  straightforward estimates on the Gaussian density.
\end{proof}

We can finally specify our choice of the barrier $E_N$ alluded to in the introduction. The optimal choice is (by far) not unique, and depends on an additional free parameter $\gamma$. The only requirement is that
\beq \label{requirement_gamma}
0<\gamma<\frac{1-\alpha}{2}\,.
\eeq
With any $\gamma$ satisfying \eqref{requirement_gamma}, and $U_N$ as in Lemma \ref{uenvelope}, we set
\beq\label{barri} 
E_{N}(k) \defi U_N(k) -N^{\gamma}\mathbf{1}_{k\neq 0, N^\alpha}
\eeq 

This choice of a barrier clearly satisfies the assumptions of Proposition \ref{fstmoment} and also Lemma \ref{eenvelope}. This has two fundamental consequences: first, the weak limit of the modified process $\Xi_{N, E_N}^{(\alpha)}$ and that of extremal process $\Xi_N^{(\alpha)}$ must necessarily coincide (provided one of the two exists); second, the mean of the modified process converges to the alleged limit, i.e. \eqref{claim_mean} holds with $E_N$ as a barrier. Theorem \ref{main} will thus follow as soon as we prove that  avoidance functions \eqref{claim_avoidance} also converge with the very same choice for the barrier. This will be done in the next section. Before that, we shall briefly comment on the choice \eqref{barri} of the barrier. (The discussion is intentionally informal: for details, the reader is referred e.g. to \cite{kistler}.) By Lemma \ref{uenvelope}, the path of extremal configurations (the process $k \mapsto S_k^\s$ for $\s$ s.t. $X_\s \approx a_N$) must necessarily satisfy the "$U_N$-barrier condition". As we have seen in Proposition \ref{fstmoment}, conditioning onto the terminal event turns the path into a Brownian bridge which is required to stay below $0$ during its lifespan. It is well known that in order to achieve this, the bridge will behave within good approximation as the path of its modulus, $k \mapsto - \mid S_k^\s \mid$, which is typically much lower than the shift 
$- N^\gamma \mathbf{1}_{k\neq 0, N^\alpha}$ for $\gamma < (1-\alpha)/2$ (this is the so-called entropic repulsion, see e.g. \cite{abk}). 
In other words, requiring that the paths stay below $E_N$ is no stricter requirement than asking them to stay below $U_N$. On the other hand, restricting the analysis on configurations whose paths stay below $E_N$ 
forces the expected number of correlated extremal pairs 
to vanish in the large $N$-limit: this stands crucially behind the Chen-Stein method which we implement below.

\section{Convergence of the avoidance functions}
The goal of this section is to prove \eqref{claim_avoidance}, which we recall reads 
\beq \label{claim_avoidance_bis}
\lim_{N\to \infty} \PP\left( \Xi_{N, E_N}(A) = 0 \right) = \PP\left( \Xi(A) = 0 \right), 
\eeq
where $E_N$ is given by \eqref{barri}, $A$ is any compact set,  and $\Xi$ is a Poisson point process with density $\mu(A) = \int_A e^{-\be_c x} dx /\sqrt{2\pi}$. To do so, we will use the so-called Chen-Stein method \cite[Theorem 1A]{barbour}. We begin with a warm-up computation. In what follows, we write $\mathcal E_N(\s)$ for the event that a configuration $\s$ satifies the "$E_N$-barrier condition", more precisely:
\[
\mathcal E_N(\s) \defi \left\{\omega \in \Omega: S^\sigma_k(\omega) \leq E_{N}(k),  k = 1\dots N^\alpha \right\}.
\]
Recall that for two configurations $\s, \tau \in \Sigma_N^{(\alpha)}$, we denote by $\s \wedge \tau$ their overlap, 
namely the first scale at which the two configurations do not coincide. 

\begin{lem}[Extremal pairs]\label{secmoment}
Let $A\subset\R$ be compact. With the above notations, it holds: 
\[
\E\left[\sharp\left\{\sigma,\tau:\;  \sigma\wedge\tau \neq 0,N^\alpha, \; \text{and}\;  X_\sigma-a_N \in A, \mathcal E_N(\s); \,  X_\tau-a_N \in A , \mathcal E_N(\tau) \right\}\right]=o\left(1\right), 
\]
as $N\to \infty$.
\end{lem}
It follows from Lemma \ref{secmoment} that energies of extremal configurations are, in fact, {\it independent} random variables. It will come hardly as a surprise that this feature stands behind the onset of the Poisson point process in the large $N$-limit.
\begin{proof}[Proof of Lemma \ref{secmoment}]
By linearity of the expectation, and re-arranging the ensuing sum according to the possible overlap-values, it holds: 
\beq \bea \label{secmomsum}
& \E\left[\sharp\left\{\sigma,\tau:\;  \sigma\wedge\tau \neq 0,N^\alpha, \; \text{and}\;  X_\sigma-a_N \in A, \mathcal E_N(\s); \,  X_\tau-a_N \in A , \mathcal E_N(\tau) \right\}\right]\\
& \qquad = \sum\limits_{K=1}^{N^\alpha-1} \#\left\{\left(\sigma,\tau\right)|\sigma\wedge\tau =K\right\}\PP\left(X_\sigma-a_N\in A, \mathcal E_N\left(\sigma\right),\; X_\tau-a_N\in A, \mathcal E_N\left(\tau\right)\right)
\eea \eeq
Let us focus on the probability on the r.h.s. above: since $\s$ and $\tau$ coincide up to scale $K$, by conditioning 
on the "trunk" which is shared by $\s$ and $\tau$, we get 
\beq \label{trunk}
\PP\left(X_\sigma-a_N\in A, \mathcal E_N\left(\sigma\right); X_\tau-a_N\in A, \mathcal E_N\left(\tau\right)\right) =\int\limits_{-\infty}^{E_{K,N} } (P) \times \PP\left(S^\sigma_K \in dx\right), 
\eeq 
where
\[
(P) \defi \PP\left(x+\left(S^\sigma_{N^\alpha}-S^\sigma_K\right)-a_N\in A, \mathcal E_N\left(\sigma\right); \; x+\left(S^\tau_{N^\alpha}-S^\tau_K\right)-a_N \in A, \mathcal E_N\left(\tau\right)\Big| S^\sigma_K = x\right). 
 \]
On the event appearing in $(P)$ we drop the $\mathcal E$-requirements: by independence of the paths after the "branching point", this  leads to 
\beq \label{trunk_two}
\eqref{trunk} \leq \int\limits_{-\infty}^{E_{K,N} }\PP\left(x+\left(S^\sigma_{N^\alpha}-S^\sigma_K\right)-a_N\in A\right)^2 \PP\left(S^\sigma_K \in dx\right).
\eeq
This steadily implies that the r.h.s. of  \eqref{trunk_two} is at most
\beq \bea \label{2ndmomestim} 
& \int\limits_{-\infty}^{E_{K,N} } \left[\int\limits_{A+a_N-x} \exp\left(-\frac{z^2}{2N^{1-\alpha}\left(N^\alpha-K\right)} \right) dz\right]^2  \exp\left( -\frac{x^2}{2N^{1-\alpha}K}\right) dx  \\
&  \leq 2^{-2N + KN^{1-\alpha}}\lambda\left(A\right) \times \\
& \qquad \times \sup\limits_{a\in A} \int\limits_{-\infty}^{0} \exp\left( -\frac{\left(a_N^{(\alpha)}-E_{K,N}-x+a\right)^2}{N^{1-\alpha}\left(N^\alpha-K\right)}-\frac{\left(x+E_{K,N}\right)^2}{2N^{1-\alpha}K} +\left(2N - KN^{1-\alpha}\right) \ln 2  \right)dx, 
\eea \eeq
where $\lambda$ denotes Lebesgue measure.  The argument of the exponential in \eqref{2ndmomestim} is easily seen to be bounded by $\beta_c\left(3\ln N- N^\gamma +x-2a\right)$, hence
\beq \bea \label{bleah}
\eqref{trunk} & \leq  2^{-2N + KN^{1-\alpha}} \lambda\left(A\right) \sup\limits_{a\in A} \int\limits_{-\infty}^{0} \exp\Bigg[ \beta_c\left(3\ln N-N^\gamma+x-2a\right)\Bigg] dx \\
& \leq 2^{-2N + KN^{1-\alpha}} C_{A} \exp\Big[ \beta_c\left(3\ln N - N^\gamma  \right)\Big],
\eea \eeq
with $C_A \defi \frac{1}{\beta_c}\lambda\left(A\right) \exp\big[-2\beta_c \inf\left\{A\right\}\big]$. Plugging \eqref{bleah} into \eqref{secmomsum}, and using that 
\[
\#\left\{\left(\sigma,\tau\right)|\sigma\wedge\tau =K\right\} \times 2^{-2N + KN^{1-\alpha}} \leq 1, 
\] 
we get 
\[
\eqref{secmomsum} \leq \sum\limits_{K=1}^{N^\alpha-1} C_{A} \exp\Big[ \beta_c\left(3\ln N - N^\gamma  \right)\Big],
\]
which is evidently vanishing in the large $N$-limit. 
\end{proof}

We can now finally move to the last missing piece, namely a proof of convergence of the avoidance functions \eqref{claim_avoidance_bis}. As mentioned, the main technical device here will be the so-called Chen-Stein method, \cite[Theorem 1A]{barbour}. To implement this, 
we need to introduce some notation. For compact $A\subset \R$, we shorten 
\[
\mu_N(A) \defi \E\left[ \Xi_{N, E_N}(A) \right]\, 
\]
and denote by $\mathcal L_{N}(A)$ the {\it law} of the random variable $\Xi_{N, E_N}(A)$. 
For a (sigma-finite) measure $\nu$ on $\R$, we denote by $\text{Pois}_{\nu(A)}$ the law of a Poisson random variable with mean $\nu(A)$. For $\rho, \rho' \in \mathcal M_1(\R)$ two probability measures on $\R$ we denote by 
$d_{TV}(\rho, \rho')$ their distance in total variation. In order to closely stick to the notation in  \cite{barbour}, we write 
\[
\Xi_{N, E_N}(A) =  \sum_{\sigma\in\Sigma_N^{(\alpha)}} I_{\sigma},  \quad 
I_\s \defi \boldsymbol{1}_{\{X_\sigma-a_N \in A, \mathcal E_N(\s)\}}, 
\]
and define, for given $\s\in \Sigma_N$, 
\[
Z_\s \defi  \sum\limits_{\tau\in\Sigma_N,\tau\wedge\sigma\neq 0,N^\alpha} I_\sigma \,.
\]
For a last piece of notation, we shorten $p_\s \defi \E[I_\s]$. 

Coming back to our main task of proving  \eqref{claim_avoidance_bis}, with $\mu(A)$ as in Theorem \ref{main}, it holds: 
\beq \bea \label{triangle}
&\Big| \PP\left( \Xi_{N, E_N}(A) = 0 \right) - \PP\left( \Xi(A) = 0 \right) \Big| \leq d_{TV}\left( \mathcal L_N(A), \text{Pois}_{\mu(A) }\right) \\
& \qquad \leq d_{TV}\left( \mathcal L_N(A), \text{Pois}_{\mu_N(A)} \right) + 
d_{TV}\left( \text{Pois}_{\mu_N(A)}, \text{Pois}_{\mu(A)} \right) 
\eea \eeq
The convergence of $\mu_N(A)$ towards $\mu(A)$ is guaranteed by Proposition \ref{fstmoment}; in virtue of simple properties of Poisson random variables, this convergence implies that  the second term on the r.h.s. above vanishes in the limit of large $N$.  Concerning the first term on the r.h.s. of \eqref{triangle}: the Chen-Stein method \cite[Theorem 1A]{barbour} yields the bound 
\beq \bea \label{chen}
d_{TV}\left( \mathcal L_N(A), \text{Pois}_{\mu_N(A)} \right) & \leq \sum_{\sigma\in\Sigma_N} \Big( p_\sigma^2 + p_\sigma \E\left[Z_\sigma\right]+ \E\left[I_\sigma Z_\sigma\right] \Big) 
\eea \eeq
Since for any $\s\in \Sigma_N$, $p_\s = 2^{-N} \mu_N(A)$, one immediately gets
\beq \label{chen_one}
\sum_{\sigma\in\Sigma_N}  p_\sigma^2 = 2^{-N} \mu_N(A)^2,
\eeq
and by simple counting, 
\beq \label{chen_two}
\sum_{\sigma\in\Sigma_N} p_\sigma \E\left[Z_\sigma\right] \leq 2^{-N^{1-\alpha}} \mu_N(A)^2. 
\eeq
Plugging \eqref{chen_one} and \eqref{chen_two} in \eqref{chen} we get 
\beq \bea \label{chen_three}
d_{TV}\left( \mathcal L_N(A), \text{Pois}_{\mu_N(A)} \right) & 
\leq 2^{-N} \mu_N(A)^2 + 2^{-N^{1-\alpha}} \mu_N(A)^2 + \sum_{\s \in \Sigma_N^{(\alpha)}} \E\left[I_\sigma Z_\sigma\right] \\
& = 2^{-N} \mu_N(A)^2 + 2^{-N^{1-\alpha}} \mu_N(A)^2 + \sum_{\s \wedge \tau\neq 0,N^\alpha} \E\left[I_\sigma I_\tau \right], 
\eea \eeq 
the last equality by definition of $Z_\s$. Since $\mu_N(A)$ converges to a finite limit (by Proposition \ref{fstmoment}), the first two terms in the last display of \eqref{chen_three} vanish in the limit of large $N$; the third term is exactly what was analyzed in Lemma \ref{secmoment}, and therefore also vanishing.  All in all, \eqref{triangle} is  vanishing,  hence \eqref{claim_avoidance_bis} holds and the proof of Theorem \ref{main} is concluded. 

\section*{Appendix}
A fundamental ingredient in the proof of Theorem \ref{main} are the estimates \eqref{bbbridge} on Brownian bridge probabilities appearing in the proof of Proposition \ref{fstmoment}; these are somewhat classical \cite{willy}, sometimes going under the name of "ballot theorems". For the reader's convenience, we give here a short proof of the estimates as needed in our framework.  

\begin{lem} \label{bbexa}
Let $\left(\Delta_i\right)_{i\in\left\{0,..,n-1\right\}}$ be i.i.d random variables having a density with respect to the Lebesgue measure and $(B_n(j), j\in\{1,..,n\})$ the related {\it bridge}, i.e.
$$ B_{n}(j) \defi \sum\limits_{i=0}^{j-1} \Delta_i - \frac{j}{n} \sum\limits_{i=0}^{n-1} \Delta_i,$$
then it holds:
\beq \label{lesszero}
 \PP\left( B_n(j) \leq 0 \text{ for all } j\in \left\{1,..,n\right\}\right)= \frac{1}{n}
\eeq
\end{lem}
\begin{proof}
We refer to $\left(\Delta_i\right)_{i\in\left\{0,..,n-1\right\}}$ as {\it increments}. The event in \eqref{lesszero} 
is equivalent to the maximum of the bridge being lower than zero. Let $m\in \{0,..,n-1\}$ be the position of the maximum; remark that this is almost surely unique by the density-assumption. One steadily checks that applying a cyclic permutation, say $\pi$, to the increments of the bridge, shifts the position of the maximum to $\pi^{-1}m$. There is one cyclic permutation only, say $\hat{\pi}$, which shifts the position of the maximum to the origin, i.e. for which $\hat{\pi}^{-1}m=0$. On the other hand, the distribution of the bridge is not affected by any permutation, hence $\hat \pi$ must be uniformly distributed among the $n$ possible cyclic permutations: since the event in \eqref{lesszero} is equivalent to $\hat{\pi}$ being the identity, the Lemma follows. 
\end{proof}

In other words, the probability that a discrete bridge stays below zero during its lifetime decays as the inverse of the length of the bridge. On the other hand, since our bridges have "square-root fluctuations", one expects that whether the bridge is required to stay below zero or below a straight line shouldn't alter (much)
the asymptotic behavior of these probabilities. This is indeed the case: 

\begin{lem}\label{bbdif}
Let $\left(B_n(j), j= \dots n\right)$ be a (discrete) Brownian bridge of length $n$. Then there exists $c>0$ independent of $n$ such that for any $n$ and $|\varepsilon|\leq c^{-1}$, it holds: 
\[
\Big| \PP\left(B_n(j )\leq 0,  j = 1 \dots n-1 \right)-\PP\left(B_n(j)\leq \varepsilon,   j = 1 \dots n-1 \right) \Big| \leq  c\frac{|\varepsilon|}{n}.
\]
\end{lem}

\begin{proof}
We proceed by induction on the length of the bridge. \\

\noindent {\sf Base case.} For $n=2$, it clearly holds: 
\[
\PP\left(B_n(j) \leq \varepsilon,  j\in\left\{1\right\},\exists j\in\left\{1\right\}: B_n(j)>0\right) = \PP\left(B_n(1) \in \left[0,\varepsilon\right]\right)\leq \frac{2}{\sqrt{\pi}}\frac{|\varepsilon|}{n} \quad \text{ for } \varepsilon>0.
\]
\[
\PP\left(B_n(j) \leq 0,  j\in\left\{1\right\},\exists j\in\left\{1\right\}: B_n(j)>\varepsilon\right) = \PP\left(B_n(1) \in \left[\varepsilon,0\right]\right)\leq \frac{2}{\sqrt{\pi}}\frac{|\varepsilon|}{n} \quad \text{ for } \varepsilon<0.
\]
The proof in the cases $\varepsilon>0$ and $\varepsilon<0$ are similar, we thus consider only the first case. \\

\noindent {\sf Induction step.}   For  $n\geq 3$, assume the claim is true for all $k \leq n-1$. By Markov inequality, 
\beq \bea \label{ind_one}
& \PP\left(B_n(j) \leq \varepsilon,  j =1 \dots n-1 \; \text{but}\;  \exists i=1 \dots n-1 : B_n(i)>0\right) \\
& \leq \sum\limits_{i=1}^{n-1}\PP\left(B_n(j)\leq \varepsilon, j =1 \dots n-1 \; \text{but}\; B_i>0\right) \\
& = \sum\limits_{i=1}^{n-1}\int\limits_{0}^{\varepsilon} \PP\left(B_n(j) \leq \varepsilon,\;  j =1 \dots n-1 \Big| \; B_n(i) = x \right)\PP\left(B_n(i) \in dx\right).
\eea \eeq
By the Markov property of Brownian bridges, the conditional probability above reads 
\[ \bea \label{markov_p}
& \PP\left(\forall_{ j\in\left\{1,..,i\right\}}:\;  B_n(j) \leq \varepsilon  \Big|\; B_n(i) = x \right)\PP\left(\forall_{  j\in\left\{i,..,n-1\right\} }\; B_n(j)\leq \varepsilon \Big|B_n(i) = x \right) \\
& \qquad \leq \PP\left(\forall_{ j\in\left\{1,..,i\right\}}:\;  B_n(j) \leq \varepsilon  \Big|\; B_n(i) = 0 \right)\PP\left(\forall_{  j\in\left\{i,..,n-1\right\} }\; B_n(j)\leq \varepsilon \Big|B_n(i) = 0 \right)
\eea \]
where the inequality follows by monotonicity in $x$. Using this, and Gaussian estimates, 
\[ \bea 
&\eqref{ind_one}\leq \sum\limits_{i=1}^{n-1}\PP\left(\forall_{j\in\left\{1,..,i\right\}}:B_n(j)\leq \varepsilon|B_n(i) = 0 \right) \times \\
&\hspace{3.5cm}\times\PP\left(\forall_{j\in\left\{i,..,n-1\right\}}:\; B_n(j)\leq \varepsilon|B_n(i) = 0 \right) \frac{\varepsilon}{\sqrt{2\pi \left(i-\frac{i^2}{n}\right)}}.
\eea \]
Given that $B_n(i) = 0$, the process $\left(B_n(j), j=1 \dots n\right)$  is a Brownian bridge of length $i \leq n-1$; analogously,  the second probability involves a Brownian bridge of length $n-i$. The assumption therefore applies, and the above is at most 
\beq \bea \label{last}
& \sum\limits_{i=1}^{n-1}\left(\PP\left(\forall_{j\in\left\{1,..,i\right\}}:B_i(j)\leq 0 \right)+c\frac{\varepsilon}{i}\right)\times \\
& \hspace{3.5cm} \times\left(\PP\left(\forall_{j\in\left\{1,..,n-i\right\}}:B_{n-i}(j)\leq 0 \right)+c\frac{\varepsilon}{n-i}\right) \frac{\varepsilon}{\sqrt{2\pi \left(i-\frac{i^2}{n}\right)}}.
\eea \eeq
It then holds:
\[\bea
\eqref{last} & \leq \varepsilon\frac{4\sqrt{n}}{\sqrt{2\pi}} \sum\limits_{i=1}^{n-1} \left(\frac{1}{i\left(n-i\right)}\right)^{3/2}\leq \varepsilon\frac{8\sqrt{n}}{\sqrt{2\pi}} \sum\limits_{i=1}^{\lfloor \frac{n}{2} \rfloor} \left(\frac{1}{i\left(n-i\right)}\right)^{3/2} 
\leq c\frac{\varepsilon}{n}\,,
\eea \]
the last inequality using the bound $[i(n-i)]^{-3/2} \leq [i (n/2)]^{-3/2}$, for $i\leq \lfloor n/2 \rfloor$.
\end{proof}

\end{document}